\documentclass{amsproc}

\usepackage{amsfonts,amscd}
\usepackage{amssymb}

\usepackage[english]{babel}

\newtheorem{theorem}{Theorem}[section]
\newtheorem{proposition}[theorem]{Proposition}
\newtheorem{lemma}[theorem]{Lemma}

\theoremstyle{definition}
\newtheorem{definition}[theorem]{Definition}

\theoremstyle{remark}
\newtheorem{remark}[theorem]{Remark}

\numberwithin{equation}{section}

\def \R{{\mathbb R}}

\def \s {{\mathbb S}}
\def \h {{\mathbb H}}

\usepackage{color}

\def \link {~}
\def \1 {\`}

\numberwithin{equation}{section}

\begin{document}

\title[Triharmonic curves in the $3$-dimensional Sol space]{Triharmonic curves in the $3$-dimensional Sol space}

\author{S.~Montaldo}
\address{Universit\`a degli Studi di Cagliari\\
Dipartimento di Matematica e Informatica\\
Via Ospedale 72\\
09124 Cagliari, Italia}
\email{montaldo@unica.it}

\author{A.~Ratto}
\address{Universit\`a degli Studi di Cagliari\\
Dipartimento di Matematica e Informatica\\
Via Ospedale 72\\
09124 Cagliari, Italia}
\email{rattoa@unica.it}

\thanks{The authors S.M. and A.R. are members of the Italian National Group G.N.S.A.G.A. of INdAM. The work was partially supported by the Project {ISI-HOMOS} funded by Fondazione di Sardegna and by PNRR e.INS Ecosystem of Innovation for Next Generation Sardinia (CUP F53C22000430001, codice MUR ECS00000038)}

\subjclass[2020]{Primary: 58E20. Secondary: 53A04, 53C30.}

\keywords{Triharmonic curves, $3$-dimensional Sol space, helices, Killing fields}

\begin{abstract}
The main aim of this paper is to study triharmonic curves in the $3$-dimensional homogeneous space Sol. In the first part of the paper we shall obtain a complete classification of proper triharmonic curves with constant geodesic curvature and torsion. In the final section we shall show that these triharmonic curves form a constant angle with a suitable Killing field of constant length along the curve.
\end{abstract}

\maketitle

\section{Introduction}
\textit{Harmonic maps} are the critical points of the {\em energy functional}
\begin{equation}\label{energia}
E(\varphi)=\frac{1}{2}\int_{M}\,|d\varphi|^2\,dV \, ,
\end{equation}
where $\varphi:M\to N$ is a smooth map between two Riemannian
manifolds $(M,g)$ and $(N,h)$. We refer to \cite{MR703510, MR1363513} for an introduction to this important topic. 

The study of \textit{polyharmonic maps of order} $r$, shortly \textit{$r$-harmonic maps}, was first proposed by Eells and Lemaire in \cite{MR703510}. These maps are defined as the critical points of the functionals 
\[
E_{r}(\varphi)=\frac{1}{2}\int_{M}\,|(d+d^*)^r\varphi|^2\,dV 
\]
which represent a version of order $r$ of the classical energy \eqref{energia}. We refer to \cite{MR4106647} for a detailed discussion on the definition of the functionals $E_{r}(\varphi)$.

In this context, the most widely studied case is $r=2$, where we have the so-called \textit{bienergy} functional and its critical points are known as \textit{biharmonic maps}. At present, a very ample literature on biharmonic maps is available (see \cite{MR4265170} and references therein).

More recently, several authors worked intensively on the case $ r \geq 3$. For instance, we refer to  \cite{MR4106647, MR4444191, MR4632927, MR4598081,MR2869168, MR3403738, MR3371364, MOR-Israel, MR3711937, MR3790367} for background and results.

In this paper we shall focus on the study of \textit{$3$-harmonic curves} in the $3$-dimensional solvable Lie group Sol. 

More generally,  let $\gamma:I \subset \R \to N$ denote {a smooth curve} parametrized by the arc length $s$ and let $T$ denote its unit tangent vector. Then $\gamma$ is $r$-harmonic if and only if (see \cite{MR4542687,MR2869168,MR4308322})
\begin{equation}\label{r-harmonicity-curves}
\nabla^{2r-1}_T T+ \sum_{\ell=0}^{r-2}(-1)^\ell R\left (\nabla^{2r-3-\ell}_T T ,\nabla^{\ell}_T T\right )T=0 \,,
\end{equation}
where $\nabla^{0}_T T=T,\,\nabla^{k}_T T=\nabla_T \left ( \nabla^{k-1}_T T\right ) $ and $R$ is the Riemannian curvature operator of $N$.

We point out that any geodesic is trivially $r$-harmonic for all $r \geq 1$. Therefore, we say that $\gamma$ is a \textit{proper} $r$-harmonic curve if it is $r$-harmonic and \textit{not} a geodesic.

The Chen-Maeta conjecture states that any $r$-harmonic submanifolds of $\R^n$ is minimal. In this direction, Maeta (see \cite{MR2869168}) proved that the conjecture is true for curves. Note that, in order to avoid trivialities arising from re-parametrization, it is convenient to restrict the attention to curves parametri\-zed by the arc length, and we shall adopt this assumption throughout the whole paper. 

By contrast, when the target is an $n$-dimensional sphere, examples of proper $r$-harmonic curves do exist. We refer to the paper of Branding \cite{MR4542687} and references therein for an illustration of the case of $r$-harmonic curves into space forms.

An investigation of triharmonic curves in $3$-dimensional homogeneous spaces with a $4$-dimensional group of isometries was carried out in \cite{MR4308322}.

As a successive step, in this paper we undertake the study of triharmonic curves in the $3$-dimensional homogeneous geometry with isometries group of dimension $3$, i.e., the $3$-dimensional solvable Lie group Sol. A study of triharmonic curves in the $3$-dimensional solvable Lie group Sol was initiated in \cite{MR4424860}, where the author gives the explicit relevant differential equations.

It is natural to start this investigation by considering a special class of geometrically significant curves, the so-called helices. 

Classically, in the Euclidean $3$-dimensional space $\R^3$, a \textit{general helix} is a curve characterized by the property that its tangent vector
makes a constant angle with a fixed straight line which represents the axis of the helix.

A famous result, stated by M. A. Lancret in $1802$ and first proved by B. de Saint Venant in $1845$, is: \textit{A necessary and sufficient condition that
a curve be a general helix is that the ratio of curvature to torsion be constant}.

Helices arise in nanosprings, carbon nanotubes and the double helix shape is commonly associated with the structure of DNA. 

In a Riemannian geometric context, the definition of a general helix is more delicate and some new interesting phenomena appear (see \cite{MR1363411}). For instance, for general helices according to \cite{MR1363411},  in the $3$-dimensional sphere $\s^3$ the ratio of the curvature and the torsion is not necessarily constant.

In the first part of the paper we follow the approach and the notation of \cite{MR4308322} and say that a  Frenet curve $\gamma(s)$ is a \textit{helix} if its geodesic curvature $\kappa(s)$ and its torsion $\tau(s)$ are constant. 
In Theorem\link\ref{Th-main-Existence} we shall obtain a \textit{complete classification of proper $3$-harmonic helices in the $3$-dimensional Sol space}.

Our paper is organised as follows. In Section\link\ref{sec-preliminaries} we recall some basic facts about curves and the geometry of Sol. The main aim of this section is to make this paper reasonably self-contained. 

In Section\link\ref{sec-results} we shall state our main result.
Next, proofs will be given in Section\link\ref{Sec-proofs}. 

Finally, in Section\link\ref{Sec-helices} we shall follow \cite{MR1363411} and consider a more general definition of a helix. This approach involves Killing fields which have constant length along $\gamma$. In this context, we shall show that {the proper $3$-harmonic curve determined in Theorem\link\ref{Th-main-Existence} is an integral curve of one of these Killing fields}. 

\section{Preliminaries}\label{sec-preliminaries}
We adopt the following notation and sign convention for the Riemannian curvature tensor field $R$:
\begin{equation}\notag
 {R}(X,Y)Z={\nabla}_{X}{\nabla}_{Y}Z
-{\nabla}_{Y}{\nabla}_{X}Z-{\nabla}_{[X,Y]}Z\,, \quad  {R}(X,Y,Z,W)=\langle {R}(X,Y)W,Z \rangle \,,
\end{equation}
where $X,Y,Z,W$ are vector fields on $N$.

Let Sol $=(\mathbb{R}^{3},g_{Sol})$ denote the $3$-dimensional Riemannian manifold given by $\R^3$ endowed with the metric tensor $g_{Sol}=e^{2z}{\rm d}x^{2}+e^{-2z}{\rm
d}y^{2}+{\rm d}z^{2}$ with respect to the standard Cartesian coordinates
$(x,y,z)$. 

The space Sol is a solvable Lie group and the metric $g_{Sol}$ is left-invariant with respect to the group operation
\[
(x,y,z).(x',y',z')=\left (e^{z} x'+x,e^{-z}y'+y,z'+z \right )\,.
\] 
The group of isometries of Sol has dimension $3$ and the connected component of the identity is generated by the following three families: 
\begin{equation}\label{isometries-SOL}
\begin{array}{l}
(x,y,z) \to (x+c,y,z)\\
(x,y,z) \to (x,y+c,z)\\
(x,y,z) \to (e^{-c}x , e^{c}y,z+c) \,.
\end{array}
\end{equation}
For future use, we point out that also
\begin{equation}\label{isometries-bis-SOL}
\begin{array}{l}
(x,y,z) \to (-x,y,z)\\
(x,y,z) \to (x,-y,z) \,.
\end{array}
\end{equation}
are isometries of Sol.
Moreover, one can easily check that a global
orthonormal frame field on Sol is
\begin{equation}\notag
E_{1}=e^{-z}\frac{\partial}{\partial x},\;
E_{2}=e^{z}\frac{\partial}{\partial y},
\;E_{3}=\frac{\partial}{\partial z}.
\end{equation}
With respect to this orthonormal frame field, the Lie brackets and the
Levi-Civita connection can be easily computed as:
\begin{equation}\notag
[E_{1},E_{2}]=0, \;[E_{2},E_{3}]=-E_{2}, \;[E_{1},E_{3}]=E_{1},
\end{equation}
\begin{equation}\label{Derivate-Ei-Ej}
\begin{array}{lll}
\nabla_{E_{1}}E_{1}=-E_{3},&\nabla_{E_{1}}E_{2}=0, &\nabla_{E_{1}}E_{3}=E_{1}\\
\nabla_{E_{2}}E_{1}=0,& \nabla_{E_{2}}E_{2}=E_{3},&\nabla_{E_{2}}E_{3}=-E_{2}\\
\nabla_{E_{3}}E_{1}=0,&\nabla_{E_{3}}E_{2}=0,&\nabla_{E_{3}}E_{3}=0.\\
\end{array}
\end{equation}
A further computation gives
\begin{equation}\label{curv-1}
\begin{array}{lll}
R(E_1,E_2)E_1=-E_{2},& R(E_1,E_3)E_1=E_{3},
&R(E_1,E_2)E_2=E_{1},\\R(E_2,E_3)E_2=E_{3},&
R(E_1,E_3)E_3=-E_{1},&R(E_2,E_3)E_3=-E_{2},
\end{array}
\end{equation}
and from this it is easy to deduce
\begin{equation}\label{curv-2}
\begin{array}{lll}
 R_{1212}=1,\\
R_{1313}=-1,\\
R_{2323}=-1.
\end{array}
\end{equation}
Now, let $\gamma(s)=(x(s),y(s),z(s))$ be a \textit{Frenet curve} in Sol parametrized by arc length. Then
\begin{equation}\label{T}
T={x'}e^z E_1 +{y'}e^{-z} E_2+{z'} E_3=\left ( T_1,T_2,T_3 \right )
\end{equation}
and we can associate to $\gamma$ the classical Frenet equations
\begin{eqnarray}\label{Frenet-field}\nonumber
\nabla_T T&=& \kappa \,N \\
\nabla_T N&=& -\kappa \,T+ \tau \,B \\\nonumber
\nabla_T B&=& - \tau  \,N\,, 
\end{eqnarray}
where $\{T,N,B\}$ is a global orthonormal frame field along $\gamma$ called the {\em Frenet frame field} and the functions $\kappa=\kappa(s)>0$, $\tau=\tau(s)$ are called the \textit{geodesic curvature} and the \textit{torsion} of $\gamma$ respectively. 

Next, let $V=V_1(s)E_1+V_2(s)E_2+V_3(s)E_3=\left ( V_1(s),V_2(s),V_3(s) \right )$ be any vector field along $\gamma$. Then a computation using \eqref{Derivate-Ei-Ej} and \eqref{T} shows that
\begin{equation}\label{nablaTV}
\nabla_T V=\left ( {V_1}'+T_1 V_3,{V_2}'-T_2 V_3,{V_3}'-T_1 V_1+T_2V_2\right )\,.
\end{equation}

For our purposes, it will be useful to know the explicit expression of $\kappa(s)$ and $\tau(s)$ for a Frenet curve $\gamma(s)=(x(s),y(s),z(s))$. Observing that
\[
\kappa^2(s)=\langle \nabla_T T,\nabla_T T \rangle \,\quad {\rm and }\quad \tau(s)=-\langle \nabla_T B, N \rangle
\]
a long, straightforward computation using \eqref{nablaTV} gives:
\begin{equation}\label{geodesiccurvaturegeneral}
\kappa^2(s)=e^{2 z} \left(x''+2 x'
   z'\right)^2+e^{-2 z} \left(y''-2 y' z'\right)^2+\left(z''-e^{2 z} x'^2+e^{-2 z} y'^2\right)^2
\end{equation}
and
\begin{equation}\label{torsionegeneral}
\tau(s)=\frac{{\mathcal A}}{{\mathcal B}},
\end{equation}
where

{\small
\begin{eqnarray*}
{\mathcal A}\hspace{-4mm}&=\hspace{-4mm}&2 e^{8 z} x'^5 y'+2 x' y'^5+e^{6 z} x' (3 x''^2
   y'+x'^2 (y^{(3)}-6 y'' z'+8 y' (2 z'^2-z''))\\
   &&+x' (y'
   (14 x'' z'-x^{(3)})-3 x'' y''))+e^{2 z} y' (x' (3 y''^2+8
   y'^2 (z''+2 z'^2)\\
   &&-y' (y^{(3)}+14 y'' z'))
  +y' (y'
   (x^{(3)}+6 x'' z')-3 x'' y''))\\
   &&+e^{4 z} (-4 x'^3 y'^3+x'
   (-y^{(3)} z''+2 y^{(3)} z'^2+z^{(3)} y''-8 y'' z'^3
  +y' (6
   z''^2+8 z'^4-4 z^{(3)} z'))\\
   &&+y' (x^{(3)} z''+2 x^{(3)}
   z'^2-z^{(3)} x''+8 x'' z'^3)+z' (y^{(3)} x''-y'' (x^{(3)}+6 x''
   z')))
   \end{eqnarray*}
and 
\begin{eqnarray*}
{\mathcal B}&=&e^{4 z} (z''^2-2 x'^2 y'^2)+e^{8 z}
   x'^4+e^{6 z} ((x''+2 x' z')^2-2 x'^2
   z'')\\
   &&+y'^4+e^{2 z} (2 y'^2 z''+(y''-2 y' z')^2)\,.
 \end{eqnarray*}
 }
As a special case of \eqref{r-harmonicity-curves} we see that $\gamma$ is triharmonic if it satisfies
\begin{equation}\label{3-harmonicity-curves}
\nabla^{5}_T T+ R\left (\nabla^{3}_T T , T\right )T - R\left (\nabla^{2}_T T ,\nabla_T T\right )T=0 \,.
\end{equation}
In the following proposition we compute explicitly the left side of \eqref{3-harmonicity-curves} and we obtain:
\begin{proposition}\label{Proptau3-explicit}Let $\gamma(s)$ be a Frenet curve in Sol parametrized by arc length. Then $\gamma$ is proper triharmonic if and only if the following three equations hold:
{\small
\begin{equation}\label{Tau3-EqT}
5 \left[2 \kappa^3 \kappa'+\kappa \left(\tau ^2 \kappa'-\kappa^{(3)}\right)-2 \kappa' \kappa''+\kappa^2 \tau  \tau
   '\right]=0 \,;
\end{equation}
\begin{equation}\label{Tau3-EqN}
\begin{array}{l}
-2 \kappa^2 \big({B_3} {T_3} \tau +5 \kappa''\big)+\kappa \big[-2 {B_3} {N_3} \tau '-15
   \kappa'^2+\big(1-2 {B_3}^2\big) \tau ^2
   -4 \tau  \tau ''-3 \tau '^2+\tau
   ^4\big]\\
   \\-4 \tau  \kappa' \big({B_3} {N_3}+3 \tau '\big)+\kappa^{(4)}-2 \kappa''
   \big(1-{B_3}^2+3 \tau ^2\big)
   +\kappa''+2 \kappa^3 \big(1-2 {B_3}^2+\tau ^2\big)+\kappa^5=0\,;
\end{array}
\end{equation}
\begin{equation}\label{Tau3-EqB}
\begin{array}{l}
-2 \kappa'' ({B_3} {N_3}-3 \tau ')+\kappa \big(2 \tau ^2 ({B_3} {N_3}-3 \tau
   ')+(2 {N_3}^2-1) \tau '+\tau ^{(3)}\big)+\kappa^3 (4 {B_3}
   {N_3}-\tau ')\\\\+\kappa^2 \tau  (2 {N_3} {T_3}-9 \kappa')+4 \kappa' \tau ''-4
   \tau ^3 \kappa'+2 \tau \big (2 \kappa^{(3)}+(2 {N_3}^2-1) \kappa'\big)=0\,.
\end{array}
\end{equation}
}
\end{proposition}
The left sides of \eqref{Tau3-EqT}, \eqref{Tau3-EqN} and \eqref{Tau3-EqB} are precisely the component of the $3$-tension field with respect to the Frenet frame field. The proof of Proposition\link\ref{Proptau3-explicit} will be carried out in Section\link\ref{Sec-proofs}.

\section{{Main Theorem}}\label{sec-results}
A simple look at Proposition\link\ref{Proptau3-explicit}, together with the general formulae \eqref{geodesiccurvaturegeneral} and \eqref{torsionegeneral}, suggest that a complete study of proper triharmonic curves in Sol may be a rather difficult task. 
Therefore, as a starting point, we restrict our investigation to a suitable family of geometrically significant curves. 
\begin{definition}\label{Def-helix} We say that a Frenet curve $\gamma(s)$ in Sol is a \textit{Frenet helix} (shortly, a \textit{helix}) if its geodesic curvature $\kappa(s)$ and its torsion $\tau(s)$ are both constant, say $\kappa(s)=a>0$, $\tau(s)=b$. 
\end{definition}
We prove a complete classification of proper triharmonic Frenet helices in Sol:
\begin{theorem}\label{Th-main-Existence} There exist a proper triharmonic Frenet helix $\gamma$ in Sol if and only if $\kappa(s)=a=1/2$ and $\tau(s)=b=\pm 1/2$. 
Moreover, up to isometries of Sol, the parametrization of such proper triharmonic helix when $b=1/2$ is
\begin{equation}\label{parametrizations}
\gamma(s)=\frac{1}{\sqrt 2}\left (-e^{-\frac{s}{\sqrt 2}},\, e^{\frac{s}{\sqrt{2}}}\,,\,\,s \right) \,.
\end{equation}
\end{theorem}
\begin{remark}In \cite{Caddeo-altri-Mediterr-2006} the authors proved that there is no proper biharmonic curve in Sol. By contrast, Theorem\link\ref{Th-main-Existence} shows that proper triharmonic curves in Sol do exist. 
\end{remark}
\section{Proofs}\label{Sec-proofs}
\begin{proof}[Proof of Proposition\link\ref{Proptau3-explicit}] We want to compute the $3$-tension field, i.e., the left side of \eqref{3-harmonicity-curves} expressing it with respect to the Frenet frame field $\{T,N,B\}$.

We shall find that, when $r=3$, the left-hand side of \eqref{Tau3-EqT}, \eqref{Tau3-EqN} and \eqref{Tau3-EqB} are precisely the components of the left-hand side of \eqref{r-harmonicity-curves} along $T$, $N$ and $B$ respectively. 

We shall use again the notation $\big ( v_1,v_2,v_3\big )=v_1 T+v_2 N + v_3 B$. First, a simple computation gives: 

\begin{eqnarray}\label{NablaiT}\nonumber
\nabla_T T&=&\big (0,\kappa,0\big )\\\nonumber
\nabla_T^2 T&=&\big (-\kappa^2,\kappa',\kappa \tau\big )\\\nonumber
\nabla_T^3 T&=&\big (-3\kappa \kappa' ,-\kappa^3-\kappa \tau^2+ \kappa'',2\kappa' \tau+\kappa \tau'\big )\\\nonumber
\nabla_T^5 T&=&\Big(5 \big(2 \kappa^3 \kappa'+\kappa \big(\tau ^2 \kappa'-\kappa^{(3)}\big)-2 \kappa' \kappa''+\kappa^2 \tau  \tau'\big),\\\nonumber
   &&\,\,\,\,\kappa^{(4)}-6 \tau ^2 \kappa''-10 \kappa^2 \kappa''-12 \tau  \kappa' \tau '+\kappa \left(\tau ^4-3 \left(5 \kappa'^2+\tau'^2\right)-4 \tau  \tau ''\right)\\
   &&+2 \kappa^3 \tau ^2+\kappa^5,\\
  & & 4 \kappa^{(3)} \tau +6 \kappa'' \tau '+4 \kappa' \tau ''-9 \kappa^2 \tau  \kappa'-4 \tau ^3 \kappa'-\kappa^3
   \tau'+\kappa \big(\tau ^{(3)}-6 \tau ^2 \tau '\big)
\Big ).  \nonumber 
\end{eqnarray}

Next, using \eqref{curv-1} and \eqref{curv-2} we compute
\begin{equation}\label{curv-3}
\begin{array}{lll}
R(T,N,T,N)&=&-1+2 B_3^2 \\
R(T,N,T,B)&=&-2 N_3 B_3 \\
R(B,T,B,T)&=& -1+2 N_3^2\\
R(B,N,N,T)&=&2 T_3 B_3 \\
R(B,N,B,T)&=&-2 T_3 N_3 \,.
\end{array}
\end{equation}
Finally, a computation using \eqref{NablaiT} and \eqref{curv-3} into \eqref{3-harmonicity-curves} leads us to complete the proof of this proposition.
\end{proof}

\begin{proof}[Proof of Theorem\link\ref{Th-main-Existence}] 

First, we establish a technical Lemma:
\begin{lemma}\label{Th-N3=0} Let $\gamma$ be a proper triharmonic Frenet helix in Sol.
Then 
\begin{itemize}
\item [\rm(i)] $N_3\equiv 0$ along $\gamma$.
\item [\rm(ii)] Both $T_3$ and $B_3$ are constant along $\gamma$.
\end{itemize}
\end{lemma}
\begin{proof}[Proof of Lemma\link\ref{Th-N3=0}] 
(i) We use Proposition\link\ref{Proptau3-explicit}. Direct substitution shows that \eqref{Tau3-EqT} always holds when $\kappa(s)\equiv a$ and $\tau(s)\equiv b$. As for equations \eqref{Tau3-EqN} and \eqref{Tau3-EqB}, they become
\begin{equation}\label{Tau3-EqNbis}
a \Big[a^4 - 2 a b B_3 T_3 + 2 a^2 (1 + b^2 - 2 B_3^2) + 
   b^2 (1 + b^2 - 2 B_3^2)\Big ]=0 
\end{equation}
and
\begin{equation}\label{Tau3-EqBbis}
2 a N_3 \big(2 a^2 B_3 + b^2 B_3 + a b T_3\big)=0
\end{equation}
respectively. Now, let us assume that $N_3(s_0)\neq 0$ for some $s_0$. Then, near $s_0$, we must have
\[
B_3=- \frac{a b T_3}{2a^2+b^2 } \,.
\]
Inserting this condition into the left side of \eqref{Tau3-EqNbis} and simplifying we find
\[
a \Big [ a^4 + b^2 + b^4 + 2 a^2 (1 + b^2)\Big ]=0 \,,
\]
which is clearly impossible since $a>0$. Then $N_3\equiv 0$. 

(ii) Since $\{T,N,B\}$ is an orthonormal frame field along $\gamma$ we deduce that $T_3^2+N_3^2+B_3^2=1$. Thus, since $N_3=0$, we must have
\[
T_3=\pm \sqrt{1-B_3^2}\,.
\]
Replacing this into \eqref{Tau3-EqNbis} we obtain
\[
a^4 + 2 a^2 (1 + b^2 - 2 B_3^2) + b^2 (1 + b^2 - 2 B_3^2) \mp 
 2 a b B_3 \sqrt{1 - B_3^2}=0\,.
\]
Then it is easy to deduce that
\[
4 a^2 b^2 B_3^2 (-1 + B_3^2) + \Big(a^4 + 2 a^2 (1 + b^2 - 2 B_3^2) + 
   b^2 (1 + b^2 - 2 B_3^2)\Big)^2=0
\]
along $\gamma$. It follows that $B_3$ is a root of a not identically zero polynomial with constant coefficients. Therefore, $B_3$ is constant along $\gamma$ and so is $T_3$.
\end{proof}
\begin{remark} Part (i) of Lemma\link\ref{Th-N3=0} was obtained by similar methods in \cite{MR4424860}, but to help the reader we have inserted here its proof.
\end{remark}

As a consequence of Lemma\link\ref{Th-N3=0}, we can assume that $N_3=0$ along $\gamma$. According to \eqref{Tau3-EqNbis}, we have to understand when there exists a Frenet helix $\gamma(s)=(x(s),y(s),z(s))$ such that
\begin{equation}\label{T3+B3=1}
 T_3^2+ B_3^2=1
\end{equation}
and
\begin{equation}\label{Tau3-EqNtris}
a^4 + 2 a^2 (1 + b^2 - 2 B_3^2) + b^2 (1 + b^2 - 2 B_3^2) - 
 2 a b B_3 \sqrt{1 - B_3^2}=0 
\end{equation}
with $T_3=\sqrt{1-B_3^2}$ constant along $\gamma$, or 
\begin{equation}\label{Tau3-EqNtris2}
a^4 + 2 a^2 (1 + b^2 - 2 B_3^2) + b^2 (1 + b^2 - 2 B_3^2) + 
 2 a b B_3 \sqrt{1 - B_3^2}=0 
\end{equation}
with $T_3=-\sqrt{1-B_3^2}$.

A routine computation shows that for any helix $\gamma$ we have:
\begin{equation}\label{T3N3B3}
\begin{array}{lll}
{\rm (i)}&T_3=&z' \\
&&\\
{\rm (ii)}&N_3=&-\displaystyle{\frac{1}{a} \big ( e^{2z}x'^2- e^{-2z}y'^2-z'' \big )}\\
&&\\
{\rm (iii)}&B_3=&\displaystyle{\frac{1}{a} \big ( -y'x''+x' (-4y'z'+y'') \big )} 
\end{array}
\end{equation}
Now, since $N_3=0$ and $T_3$ is constant, we deduce from \eqref{T3N3B3}(ii) that either
\begin{equation}\label{z1}
z(s)=z_1(s)=\frac{1}{2} \log \left (\frac{y'}{x'} \right )
\end{equation} 
or
\begin{equation}\label{z2}
z(s)=z_2(s)=\frac{1}{2} \log \left (-\frac{y'}{x'} \right )
\end{equation} 
with the obvious sign restrictions on $x',y'$. Let us first assume that $z(s)=z_1(s)$. Then
\[
T_3=z_1'= \frac{1}{2} \left (-\frac{x''}{x'}+\frac{y''}{y'} \right )=c_1 \quad\quad (|c_1|\leq1)\,.
\]
Integration of this condition yields
\begin{equation}\label{yprimoxprimo}
y'=c_2 e^{2c_1s}x'\quad (c_2>0)\,.
\end{equation}

Next, using \eqref{T}, \eqref{z1} and imposing $\langle T,T\rangle=1$ we deduce
\[
c_1^2+2 c_2 e^{2c_1s}x'^2=1 
\]
along $\gamma$. Integration of this equality tells us that either
\begin{equation}\label{x1}
x(s)=x_1(s)=-\frac{\sqrt{(1-c_1^2)}}{c_1\sqrt{2 c_2}} e^{-c_1s}+c_x 
\end{equation}
or 
\begin{equation}\label{xtilde}
x(s)=\tilde{x}_1(s)=\frac{\sqrt{(1-c_1^2)}}{c_1\sqrt{2 c_2}} e^{-c_1s}+c_x \,.
\end{equation}
First, let us assume that we are in the case \eqref{x1}. Then direct substitution into \eqref{yprimoxprimo} and integration enables us to deduce the explicit expression for $y(s)$ and, using \eqref{z1}, for $z(s)$ as well. The final output is:
\[
y(s)=y_1(s)=\frac{\sqrt{(1-c_1^2)c_2}}{c_1\sqrt{2 }} e^{c_1s}+c_y \,,\quad z_1(s)=\frac{1}{2}\log \left ( c_2  e^{2c_1s}\right ) \,.
\]
Now, using the explicit expression $(x_1(s),y_1(s),z_1(s) )$ into \eqref{geodesiccurvaturegeneral} and \eqref{torsionegeneral} we find that we must have:
\[
a^2=c_1^2-c_1^4 \,, \quad b=1-c_1^2 \,.
\]
Finally, using these values for $a,b$ into  \eqref{Tau3-EqNtris}, \eqref{Tau3-EqNtris2}, together with $B_3=\pm \sqrt{1-c_1^2}$, it is straightforward to check that the only acceptable values are $c_1=\pm 1/\sqrt 2$. Both these values for $c_1$ give $a=1/2$ and $b=1/2$. If we now choose $x=\tilde{x}_1$ the analysis is similar. We find 
\[
y(s)=\tilde{y}_1(s)=-\frac{\sqrt{(1-c_1^2)c_2}}{c_1\sqrt{2 }} e^{c_1s}+c_y \,,\quad z_1(s)=\frac{1}{2}\log \left ( c_2  e^{2c_1s}\right ) 
\]
and the conclusion is as before. By way of summary, the only possible solutions associated to the case $z(s)=z_1(s)$ are
\[
\gamma(s)=\left(x_1,y_1,z_1 \right)
\]
where
\begin{equation}\label{parametrizations-c2}
\begin{array}{l}
x_1=-\frac{e^{-\frac{s}{\sqrt{2}}}}{ \sqrt{{2c_2}}}+c_x\,,\quad y_1=\frac{e^{\frac{s}{\sqrt{2}}}\sqrt{c_2}}{ \sqrt{2}}+c_y\,,\quad z_1(s)=\frac{1}{2} \log \left(c_2 e^{\sqrt 2 s}\right) \\
{\rm or}\\
x_1=\frac{e^{\frac{s}{\sqrt{2}}}}{ \sqrt{{2c_2}}}+c_x\,,\quad y_1=-\frac{e^{\frac{-s}{\sqrt{2}}}\sqrt{c_2}}{ \sqrt{2}}+c_y\,,\quad z_1(s)=\frac{1}{2} \log \left(c_2 e^{-\sqrt 2 s}\right )\\
{\rm or}\\
x_1=\frac{e^{-\frac{s}{\sqrt{2}}}}{ \sqrt{{2c_2}}}+c_x\,,\quad y_1=-\frac{e^{\frac{s}{\sqrt{2}}}\sqrt{c_2}}{ \sqrt{2}}+c_y\,,\quad z_1(s)=\frac{1}{2} \log \left(c_2 e^{\sqrt 2 s}\right )\\
{\rm or}\\
x_1=-\frac{e^{\frac{s}{\sqrt{2}}}}{ \sqrt{{2c_2}}}+c_x\,,\quad y_1=\frac{e^{\frac{-s}{\sqrt{2}}}\sqrt{c_2}}{ \sqrt{2}}+c_y\,,\quad z_1(s)=\frac{1}{2} \log \left(c_2 e^{-\sqrt 2 s}\right )
\end{array}
\end{equation}
where $c_1,c_2,c_x,c_y$ are integration constants with $c_2>0$. Taking into account \eqref{isometries-SOL} and \eqref{isometries-bis-SOL}, we conclude that, up to isometries, the only proper triharmonic curve associated to the case $z(s)=z_1(s)$ is \eqref{parametrizations}, which is the first parametrization of \eqref{parametrizations-c2} with $c_x=0,c_y=0, c_2=1$.

Next, if we repeat the same analysis with $z(s)=z_2(s)$ as in \eqref{z2}, we find that
\[
c_1^2-c_1^4=a^2 \,, \quad b=-1+c_1^2 \,.
\]
Again, the only acceptable values are $c_1=\pm 1/\sqrt 2$. Both these values for $c_1$ give $a=1/2$ and $b=-1/2$. The parametrization of $\gamma$ in this case, up to isometries, is
\begin{equation*}\label{parametrizations-2}
\gamma(s)=\frac{1}{\sqrt 2}\left (e^{-\frac{s}{\sqrt 2}},\, e^{\frac{s}{\sqrt{2}}}\,,\,\,s \right)
\end{equation*}
which, up to the isometry $(x,y,z)\to (-x,y,z)$, coincides with \eqref{parametrizations} and so the proof is completed.
\end{proof}
\section{General helices}\label{Sec-helices}
As we pointed out in the introduction, the most suitable definition of a general helix in a Riemannian geometric context is a delicate matter and may depend on the specific situation. Definition\link\ref{Def-helix} above seemed to us the most appropriate for a first approach to the study of triharmonic curves in Sol. 

In this section we study the connection between the proper triharmonic Frenet helices provided by Theorem\link\ref{Th-main-Existence} and the notion of a \textit{general helix} according to the paper \cite{MR1363411} of Barros. In this article the author, along the lines of previous work by Langer and Singer (see \cite{MR772124}), proposes the following definition:
\begin{definition}\label{Def-Barros}{A smooth, regular curve }$\gamma(s)$ into a $3$-dimensional space form will be called a \textit{general helix} if there exists a Killing vector field $\mathcal{V }(s)$ with constant length along $\gamma$ and such that the angle between $\mathcal{V }$ and the unit tangent $T$ is a non-zero constant along $\gamma$. We will say that $\mathcal{V }$ is an axis of the general helix.
\end{definition}
Using this definition, Barros proved a version of Lancret Theorem for both $\h^3$ and $\s^3$. The spherical case is the most interesting because of the appearance of general helices whose geodesic curvature and torsion are related by a rule of the type $\tau(s)=\kappa(s)\pm 1$ (see \cite{MR1363411}).

Now, we discuss this notion of a general helix in our context. 

First, we recall that a base for the Killing fields on Sol is $\left \{\mathcal{V }_1,\mathcal{V }_2,\mathcal{V }_3 \right \}$, where
\begin{equation}\label{Killing-fields}
\mathcal{V }_1=e^z E_1\,,\quad \mathcal{V }_2=e^{-z} E_2\,,\quad \mathcal{V }_3=-x e^z E_1+y e^z E_2+E_3\,.
\end{equation}
Our result is:
\begin{proposition}\label{Prop-general-helices} 

{\rm (i)} There exists no proper triharmonic curve $\gamma(s)$ in Sol such that both $|\mathcal{V }_1|$ and $\beta$ are constant along $\gamma$, where $\beta$ denotes the angle between $\mathcal{V }_1$ and the unit tangent vector field $T$. The same is true for $\mathcal{V }_2$.

{\rm (ii)} Let $\gamma(s)$ be a proper triharmonic curve in Sol as in Theorem\link\ref{Th-main-Existence}. Then $|\mathcal{V }_3|=\sqrt 2$ and $\beta=0$ along $\gamma$, where $\beta$ is the angle between $\mathcal{V }_3$ and $T$. Then $\gamma$ is an integral curve of $\mathcal{V }_3\slash \sqrt 2$.
\end{proposition}
\begin{proof}(i) We argue by contradiction. We observe that $|\mathcal{V }_1|^2=e^{2z(s)}$. Therefore, we must have $z(s)=c$ along $\gamma$. Then explicit integration of the conditions $\langle T,T \rangle =1$ and $\langle T,\mathcal{V }_1 \rangle =|\mathcal{V }_1|\cos \beta$ tells us that $\gamma(s)$ must be of the type
\[
\gamma(s)=(\cos \beta e^{-c} s+c_x,\sin \beta e^{c} s+c_y,c)\,.
\]
Then a computation using \eqref{geodesiccurvaturegeneral} and \eqref{torsionegeneral} shows that, for this type of curve, $\kappa(s)=\sqrt{\cos^2(2 \beta)}$, $\tau(s)=\sin (2\beta)$. Moreover, equations \eqref{Tau3-EqT} and \eqref{Tau3-EqB} are verified. But \eqref{Tau3-EqN} becomes equivalent to
\[
\sqrt{\cos^2(2 \beta)} (5+ \cos (4 \beta))=0 \,.
\]
This is a contradiction since $\cos(2 \beta)=0$ is not acceptable because $\gamma$ is not a geodesic. The case of $\mathcal{V }_2$ is analogous.

(ii) We use the parametrization \eqref{parametrizations} of $\gamma$. Computing we find that along $\gamma$
\[
|\mathcal{V }_3|^2=x^2 e^{2 z}+y^2 e^{-2 z}+1=2 \,;\quad |\mathcal{V }_3|\cos \beta=  \langle T,\mathcal{V }_3\rangle = -x x' e^{2z} + y y' e^{-2z} +z'=\sqrt 2 
\]
so that $|\mathcal{V }_3|=\sqrt 2$ and $\beta=0$. {Then $\gamma$ is an integral curve of $\mathcal{V }_3\slash \sqrt 2$, as required to end the proof}.
\end{proof}

%\bibliographystyle{siam}
%\bibliography{MORbib.bib}

\bibliographystyle{amsalpha}

\end{document}